\documentclass[11pt,a4paper]{amsart}
\usepackage{graphicx}
\usepackage{subcaption}
\usepackage{setspace,mathrsfs,amsthm,amsmath,amssymb,amsfonts,enumerate,lipsum,appendix,mathtools,cite,tabularx,fontenc}
\usepackage{tikz,float}
\usepackage{subcaption}

\makeatletter
\newtheorem*{rep@theorem}{\rep@title}
\newcommand{\newreptheorem}[2]{%
	\newenvironment{rep#1}[1]{%
		\def\rep@title{#2 \ref{##1}}%
		\begin{rep@theorem}}%
		{\end{rep@theorem}}}
\makeatother

\newtheorem{theorem}{Theorem}[section]
\newreptheorem{theorem}{Theorem}
\newtheorem{lemma}{Lemma}[section]

\newtheorem{proposition}[lemma]{Proposition}
\newtheorem{remark}[lemma]{Remark}
\newtheorem{definition}[lemma]{Definition}

\usepackage[nointegrals]{wasysym}
\usepackage[misc]{ifsym}
\definecolor{ao}{rgb}{0.0, 0.5, 0.0}
\definecolor{lasallegreen}{rgb}{0.03, 0.47, 0.19}
\usepackage{hyperref}
\hypersetup{
	colorlinks=true,
	linkcolor=blue,
	filecolor=magenta,      
	urlcolor=cyan,
	citecolor=red,
}
\usepackage[margin=0.85in]{geometry}

\let\oldnorm\norm
\def\norm{\@ifstar{\oldnorm}{\oldnorm*}}

\newcommand{\al} {\alpha}
\newcommand{\pa} {\partial}

\newcommand{\De} {\Delta}

\newcommand{\ga} {\gamma}
\newcommand{\Ga} {\Gamma}

\newcommand{\Om} {\Omega}
\newcommand{\la} {\lambda}

\newcommand{\noi} {\noindent}
\newcommand{\lmap}{\longmapsto}

\newcommand{\dx}{{\,\rm d}x}
\newcommand{\dt}{{\,\rm d}t}

\newcommand{\dS}{{\,\rm dS}}

\newcommand{\C}{{\mathcal C}}
\newcommand{\E}{{\mathcal E}}
\newcommand{\F}{{\mathcal F}}

\newcommand{\R}{{\mathbb R}}
\newcommand{\N}{{\mathbb N}}

\newcommand\Item[1][]{%
	\ifx\relax#1\relax  \item \else \item[#1] \fi
	\abovedisplayskip=0pt\abovedisplayshortskip=0pt~\vspace*{-\baselineskip}}

\newcommand{\wide}[1]{\widetilde{#1}}

\def\w{{\widetilde w}}

\def\dx{{\,\rm d}x}
\def\dt{{\,\rm d}t}

\def\sb2{{{\mathcal D}^{1,2}_0(B_1^c)}}

\def\w2r{{{ W}^{2,2}(\R^N)}}

\def\d2{{{\mathcal D}^{2,2}_0(\Om )}}

\def\A{{\mathcal A}}

\def\C{{\mathcal C}}

\def\E{{\mathcal E}}

\def\Nn{{\mathcal N}}

\def\F{{\mathcal F}}

\def\h1{{H^1(\Om)}}

\def\ws2{{\F_{\frac{N}{2}}}}

\def\c1Loc{{\C_{loc}^1}}

\usepackage[hyperpageref]{backref}

\makeatletter
\newcommand{\setword}[2]{%
  \phantomsection
  #1\def\@currentlabel{\unexpanded{#1}}\label{#2}%
}

\begin{document}
\singlespacing
\title[Optimal Shapes for the First Dirichlet Eigenvalue]{Optimal Shapes for the First Dirichlet Eigenvalue of the $p$-Laplacian and Dihedral symmetry}

\author{Anisa M. H. Chorwadwala}
\address{Address: Indian Institute of Science Education and Research Pune, Dr. Homi Bhabha Road, Pashan,
Pune 411008, India}

\email{anisa23in@gmail.com, anisa@iiserpune.ac.in,}


\author{Mrityunjoy Ghosh}
\address{Address: Department of Mathematics, Indian Institute of Technology Madras, Chennai 600036, India}
\email{ghoshmrityunjoy22@gmail.com}





\begin{abstract}
In this paper, we consider the optimization problem for the first Dirichlet eigenvalue $\lambda_1(\Omega)$ of the $p$-Laplacian $\Delta_p$, $1< p< \infty$, over a family of doubly connected planar domains $\Omega= B \setminus \overline{P}$,
where $B$ is an open disk and $P\subsetneq B$ is a domain which is invariant under the action of a dihedral group $\mathbb{D}_n$ for some $n \geq 2,\;n\in \mathbb{N}$. We study the behaviour of $\lambda_1$ with
respect to the rotations of $P$ about its center. We prove that the extremal configurations correspond to the cases where $\Omega$ is symmetric with respect to the line containing both the centers. Among these optimizing domains, the OFF configurations correspond to the minimizing ones while the ON configurations correspond to the maximizing ones. Furthermore, we obtain symmetry (periodicity) and monotonicity properties of $\lambda_1$ with respect to these rotations. In particular, we prove that the conjecture formulated in \cite{Anisa2020} for $n$ odd and $p=2$ holds true. As a consequence of our monotonicity results, we show that if the nodal set of a second eigenfunction of the $p$-Laplacian possesses a dihedral symmetry of the same order as that of $P$, then it can not enclose $P$.
\end{abstract}
\maketitle
\section{Introduction and statements of the main results}\label{intro}
Let $\Om \subset \mathbb{R}^2$ be a bounded domain. For $1<p<\infty$, we consider the following eigenvalue problem for the $p$-Laplacian on $\Om:$
\begin{equation}\tag{E}\label{problem}
	\left. \begin{aligned}
	-\Delta_p u &=\lambda |u|^{p-2}u \quad\text{in} \quad\Omega,\\
	u &=0 \qquad\quad \quad\;\text{on}\quad \partial \Om,
	\end{aligned}\right\}
\end{equation}
where $\Delta_p$ is the $p$-Laplacian given by
$\Delta_p u \coloneqq \text{div}(|\nabla u|^{p-2}\nabla u)$.

A real number $\la$ is called an eigenvalue of (\ref{problem}) if $\exists$ $u\in W^{1,p}_{0}(\Om)\setminus\{0\}$ such that the following holds:
\begin{equation*}
	\int_{\Om}|\nabla u|^{p-2}\left<\nabla u,\nabla w \right> \dx=\la\int_{\Om}|u|^{p-2}uw \dx, \quad\forall\; w\in W^{1,p}_{0}(\Om).
\end{equation*}
\par Let $\la_1(\Om)$ denote the smallest positive eigenvalue of the boundary value problem (\ref{problem}). 
Then, $\la_1(\Om)$  has the following variational characterization; cf. \cite{Azorero1987}:
\begin{equation}\label{vari}
	\la_1(\Om)=\displaystyle\inf\Bigg\{\frac{\int_{\Om}|\nabla u|^p}{\int_{\Om}|u|^p}: u\in W^{1,p}_{0}(\Om)\setminus\{0\} \Bigg\}.
\end{equation}
This eigenvalue is simple; cf. \cite{Lindqvist2006}. Without loss of generality, an eigenfunction corresponding to $\la_1(\Om)$ can be chosen to be positive and of unit $L^p$-norm; cf. \cite{Lindqvist1990}.

For $p=2$, Hersch \cite{Hersch} studied the problem (\ref{problem}) on doubly connected planar domains having holes with some assumption on the boundary. Using the \textit{effectless cut} introduced by Weinberger \cite{Weinberger}, he proved that \textit{``among all doubly or multiply connected domains which are fixed along their outer boundary $\Ga_0$ and one inner boundary $\Ga_1$ (and otherwise free) with given $A,\;l_{\Ga_0},\;l_{\Ga_1}$ satisfying $l_{\Ga_0}^2- l_{\Ga_1}^2=4\pi A$, the concentric annular membrane has the maximal $\la_1$"}, where $A$ is the area of the domain and $l_{\Ga_0},\;l_{\Ga_1}$ denotes the length of $\Ga_0$ and $\Ga_1$ respectively. In particular, if $\Om=B_r\setminus \overline{B_s}$, where $B_r,\;B_s$ are open disks in $\mathbb{R}^2$ of radius $r$ and $s$ respectively, then Hersch's result yields that, 
\begin{equation}\label{Hersch_ball}
    \la_1(B_r\setminus \overline{B_s})\leq \la_1(\Om_\#).
\end{equation}
where $\Om_\#$ is the concentric annulus bounded by the disks $B_r$ and $B_s$. However, inequality \eqref{Hersch_ball} does not give any insight about the variations of $\la_1$ with respect to the placement of $B_s$ inside $B_r$. In \cite{ramm},  Ramm and Shivakumar conjectured that $\la_1(B_r\setminus \overline{B_s})$ strictly decreases when the inner disk $B_s$ moves towards the boundary of the outer disk $B_r$. This conjecture was proved analytically later in the article (arXiv:math-ph/9911040) using an argument communicated to Ramm by Ashbaugh. Later, Kesavan in  \cite{Kesavan} and Harell in \cite{Harell2001} extended this result to higher dimensions (i.e., $n\geq 3$) for the same case $p=2$ with independent proofs. Chorwadwala and Mahadevan \cite{Anisa2015} generalized this monotonicity result for the first eigenvalue $\la_1$ of \eqref{problem}. They showed that \textit{``$\la_1(B_r\setminus \overline{B_s})$ decreases when the center of the inner ball $B_s$ moves towards the boundary of the outer ball $B_r$ along the radial directions of $B_r$"}. However, the strict monotonicity has been obtained by Anoop, Bobkov, and Sasi in \cite{Anoop2018}. The results related to the optimization of eigenvalues with respect to other boundary conditions can be found in \cite{Anoop2020Ashok,Anoop2020Kesavan,Anoop2021Szego}. For further literature and open problems in this direction, we refer to the book \cite{Henrot2006}; see also \cite{henrot2021} for recent developments in spectral geometry. 

In this article, we consider the boundary value problem \eqref{problem} on doubly connected planar domains whose outer boundary is a circle, whereas the inner domain possesses a dihedral symmetry. First, we start with establishing a notational setup.
\begin{description}
\item[A0\label{itm:A0}]
Let $P$ be a bounded and simply connected planar domain such that the boundary $\pa P$ is a simple, closed $C^2$ curve. 
\end{description}

For $n\in \N$, $n \geq 2$, consider the dihedral group $\mathbb{D}_n$ of order $2n$ generated by the rotation $\rho$ by angle $\frac{2\pi}{n}$ and a reflection $s$ such that $s\rho s = \rho^{-1}$.
We make further assumptions (\nameref{itm:A1}) and (\nameref{itm:A2}) on the domain $P$ satisfying (\nameref{itm:A0}) as follows.
\begin{description}
\item[A1\label{itm:A1}] ($\mathbb{D}_n$ symmetry) $P$ has a $\mathbb{D}_n$ symmetry for some $n \in \N$, $n \geq 2$.
\end{description}
Assumption (\nameref{itm:A1}) means that $P$ is invariant under the action of the dihedral group $\mathbb{D}_n$ of order $2n$.
Such domains admit $n$ axes of symmetry (i.e., $2n$ half-axes of symmetry), wherein two consecutive axes make an angle $\frac{\pi}{n}$ with each other. Also, there exists a unique point $o$ inside $P$, which we call as the center of $P$, at which all these axes of symmetry of $P$ intersect. Assuming that $o=(0,0)$, let the boundary $\pa P$ of $P$ be represented (See \eqref{P_parametrization}) in polar coordinates as $(f(\phi),\phi)$, where $\phi\in [0,2\pi)$ and $f:[0,2\pi)\longmapsto \R$ is a $C^2$ map with $f(0)=f(2\pi)$.
\begin{description}
\item[A2\label{itm:A2}] (Monotonicity of the boundary $\partial P$)
 Consider the sectors formed by two consecutive axes of symmetry of the domain $P$. On each of these sectors, the distance $d(o,x)=\|x\|$ for $x
\in \pa P$ is monotonic as a function of $\phi$, i.e., on each sector formed by two consecutive axes of symmetry of $P$, the function $x \lmap d(o,x)$, thought of as the function $\phi \lmap f(\phi)$, is a monotonic function.
\end{description}
 \begin{definition}
  The  points $(f(\phi), \phi)$ on $\pa P$ at which $f(\phi)$ is maximum are called \textit{Outer Vertices} of $P$. Likewise, the  points on $\pa P$ at which $f(\phi)$ is minimum are called the \textit{Inner Vertices} of $P$. The union of the outer vertices and inner vertices of $P$ will be referred to as the Vertices of $P$.
  \end{definition}
  
 From the above definition, it is easy to observe that if $P$ satisfies (\nameref{itm:A0})-(\nameref{itm:A2}), then it has a total of $2n$ vertices, out of which $n$ vertices are inner vertices while the other $n$ vertices are outer vertices. 
 
 In this article, we consider the domains satisfying the following:
  \begin{description}
\item[D\label{itm:Domain}]
 $\Om= B \setminus \overline{P}\subset \mathbb{R}^2$,
where $B$ is an open disk, $P$ is a domain satisfying (\nameref{itm:A0}), (\nameref{itm:A1}) and (\nameref{itm:A2}) such that  $\rho_t(P)\subset B$, $\forall \;t \in [0,2\pi)$, where $\rho_t$ denotes the rotation in $\R^2$ about the center $o$ of $P$ in the anticlockwise direction by angle $t$. \end{description}

Since the first Dirichlet eigenvalue $\la_1(\Om)$ associated to the eigenvalue problem (\ref{problem}) is invariant under isometries of $\mathbb{R}^2$, without loss of generality, we assume that
the centers of $B$ and $P$ are on the $x_1$-axis. As supposed earlier, the center $o$ of $P$ is at the origin, i.e., $o=(0,0)$. Let the center of $B$ be $(-c,0)$ for $c\geq 0$.
Now we give the following definitions concerning the axial symmetric position of a domain $\Om$ satisfying (\nameref{itm:Domain}); see \cite[Section 3.2]{Anisa2020} for an equivalent statement:
\begin{definition}\label{OnOff_position}
The domain $P$ is said to be in an OFF position with respect to the disk $B$ if an inner vertex of $P$ lies on the negative $x_1$-axis. The domain $P$ is said to be in an ON position with respect to $B$ if an outer vertex of $P$ lies on the negative $x_1$-axis.
\end{definition}

\begin{figure}[H]\centering
\subfloat[OFF
]{
\label{Off position}
\includegraphics[width=0.15\textwidth]{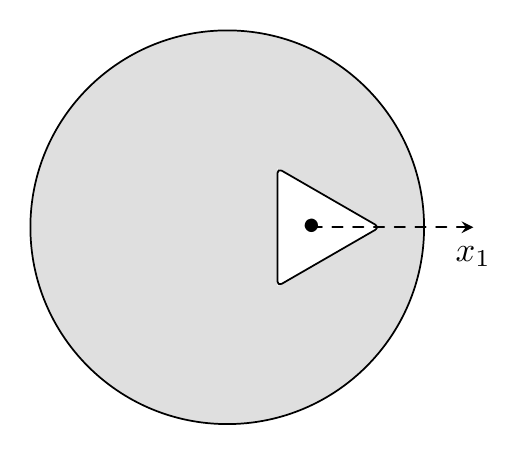}}
\hspace{10mm}
\subfloat[ON
]{
\label{On position}
\includegraphics[width=0.155\textwidth]{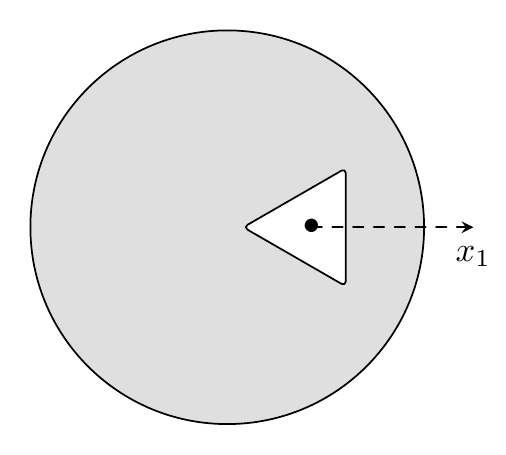}}
\hspace{10mm}
\subfloat[OFF
]{
\label{Off position2}
\includegraphics[width=0.15\textwidth]{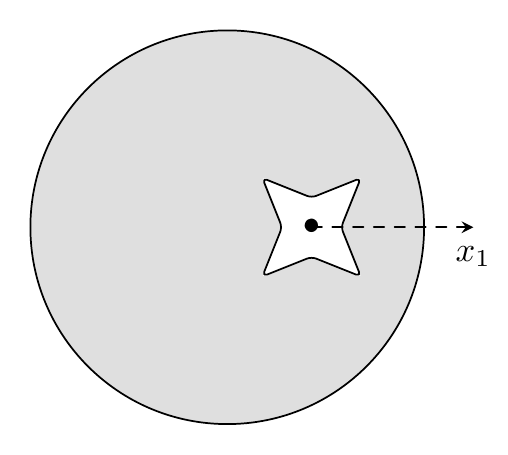}}
\hspace{10mm}
\subfloat[ON
]{
\label{On position2}
\includegraphics[width=0.15\textwidth]{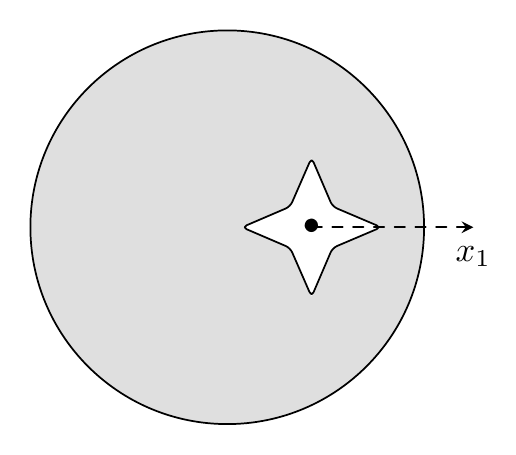}}
\caption{OFF and ON configurations: $P$ having $\mathbb{D}_3$ symmetry-(A) and (B),\\ $P$ having $\mathbb{D}_4$ symmetry-(C) and (D). }\label{fig:on_off}
\end{figure}

In \cite[Theorem 1]{kiwanDihed}, for $p=2$ and $\Om=P_1\setminus \overline{P_2}$, where $P_1$ and $P_2$ satisfy (\nameref{itm:A0}), (\nameref{itm:A1}) (for fixed $n\in\N,\;n\geq 2$), (\nameref{itm:A2}) and $P_1$, $P_2$ concentric, authors proved that among all rotations of the inner domain $P_2$ about its center, $\la_1(\Om)$ takes maximum (minimum) value when $P_2$ is in an ON (OFF) position with respect to $P_1$. The proof of this result is based on the Hadamard perturbation formula (See Section \ref{Hadamard}), which gives a representation of $\la_1'$, and a moving plane method due to Serrin \cite{Serrin}. An analogous result of \cite{kiwanDihed} for the nonconcentric case is proved by Chorwadwala and Roy \cite{Anisa2020} using a rotating plane method under additional assumptions on the outer and inner boundaries. More precisely, for $n$ even, $p=2$ and $\Om=B\setminus \overline{P}$ as in (\nameref{itm:Domain}), they proved that among all rotations of the inner domain $P$ about its center, $\la_1(\Om)$ is optimum only when an axis of symmetry of $P$ coincides with a diameter of $B$ \cite[Theorem 5.1]{Anisa2020}. Furthermore, authors obtained that $\la_1(\Om)$ attains maximum (minimum) when $P$ is in an ON (OFF) position with respect to $B$ (See Figure \ref{fig:on_off}-(C),(D)). The idea behind the proof of this result differs from those used in \cite{kiwanDihed} as the earlier techniques fail due to the nonconcentric position of the two domains $B$ and $P$. However, authors used a sector reflection technique for $n$ even which helps them to study the behaviour of $\la_1$ on upper and lower hemispheres of the disk $B$. Moreover, the assumption of having even dihedral symmetry of the inner domain $P$ leads to an appropriate pairing of sectors as, for $n$ even, the axes of symmetry of $P$ split $B$ into an even number of sectors in each
hemisphere, upper and lower. Such pairing is no longer complete when $n$ is odd, as pointed out in \cite[Section 6.7]{Anisa2020}. Nevertheless, with numerical evidence \cite[Section 8]{Anisa2020}, authors conjectured that the result, analogous to the  $n$ even case, holds for $n$ odd too. 

In this article, we give an affirmative answer to the aforementioned conjecture. For $n$ odd, we use a mixed type pairing by choosing sectors across the upper and lower hemispheres and establish that the sector reflection technique (Lemma \ref{Sector_reflection_odd}) works with this choice of pairing. This helps us to overcome the difficulty faced in \cite{Anisa2020}. Further, we extend these results to the $p$-Laplacian, $1<p<\infty$, and for all $n\in \N$, $n\geq 2$. To the best of our knowledge, the only result concerning the strict monotonicity of the first eigenvalue of the $p$-Laplacian (with respect to translations of the inner domain $P$) has been obtained in \cite{Anoop2018} when both $B$ and $P$ are balls. For us $P$ is not a ball and hence their proof does not work in our case as it is. Before going to state our first main result, we introduce the class of admissible domains.

Let $\Om$ be as described in
(\nameref{itm:Domain}) with $B$ centered at $(-c,0)$ for $c>0$. In order to study the behaviour of $\lambda_1(\Om)$ with respect to the rotations of the inner domain $P$ about its center, we consider deformations of $\Om$, $\Om \lmap \Om_t$, with the center of $P$ fixed at $o$ as described below.
\begin{equation}\label{t_configuration}
    \Omega_t:= B \setminus \overline{P_t},\;\text{ where }~ P_t:=\rho_t(P),\; t\in[0,2\pi).
\end{equation}

Now we state our first main result for $\Om_t$ as in \eqref{t_configuration}.
\begin{theorem}[Optimal configuration with respect to rotations of P]\label{theorem}
Let $n\in \N$ and $n\geq 2$. For $1<p<\infty$, consider the map $t \lmap \la_1(\Om_t)$ for $t\in [0,2\pi)$, where $\la_1(\Om_t)$ is the first eigenvalue of \eqref{problem} on $\Om_t$. Then the following holds.
\begin{enumerate}[(i)]
    \item For $\frac{3}{2}<p<\infty$, $\la_1(\Om_t)$ achieves its optimal value if and only if an axis of symmetry of $P$ coincides with a diameter of $B$. 
Moreover, $\la_1(\Om_t)$ is maximum only when $P$ is in an ON position with respect to $B$ and $\la_1(\Om_t)$ is minimum only when $P$ is in an OFF position with respect to $B$. Furthermore, between the ON and the OFF configuration, $t \lmap \la_1(\Om_t)$ is a strictly decreasing function.
\item For $1<p\leq \frac{3}{2}$, $\la_1(\Om_t)$  achieves its optimal value when an axis of symmetry of $P$ coincides with a diameter of $B$. Furthermore, $\la_1(\Om_t)$ will be maximum if $P$ is in an ON position with respect to $B$ and $\la_1(\Om_t)$ will be minimum if $P$ is in an OFF position with respect to $B$. Moreover, between the ON and the OFF configuration, $t \lmap \la_1(\Om_t)$ is a decreasing function. 
\end{enumerate}
\end{theorem}

The main ingredients of our proof are $(a)$ the strong comparison principle due to Sciunzi \cite{Sciunzi2014} for $\frac{3}{2}<p<\infty$ and $(b)$ the weak comparison principle due to Chorwadwala, Mahadevan, and Toledo \cite{Anisa2015Toledo} available for $1< p<\infty$ which we use for $1<p\leq \frac{3}{2}$. It is worth mentioning here that the application of the strong comparison principle mentioned above gives a direct proof for the strict monotonicity of $\la_1$ for $\frac{3}{2}<p<\infty$. However, for $1<p\leq \frac{3}{2}$, we obtain a weaker result only, giving the non-strict monotonicity of $\la_1$.

\subsection*{Application:} In this section, we study the structure of the nodal set of a second eigenfunction of the $p$-Laplacian. First, we give the variational characterization of the second eigenvalue, cf. \cite{Anoop2016Drabek,Azorero1987,Drabek1999Robinson}; see \cite{Drabek2020Robinson} for recent developments on the eigenvalues of the $p$-Laplacian. Let $\Om\subset \R^2$ be a bounded domain. Consider the following functionals on $W^{1,p}_0(\Om)$:
\begin{equation}\label{Functional}
    J(u)=\int\limits_{\Om}|\nabla u|^p\dx\;;\; G(u)=\int\limits_{\Om}|u|^p\dx.
\end{equation}
Let $\mathcal{E}=G^{-1}(1)$ and $\mathcal{F}=\{\mathcal{A}\subset \mathcal{E}: \mathcal{A}=\ga(\mathcal{S}^1),\;\ga:\mathcal{S}^1\longmapsto \mathcal{E}\; \text{is an odd and continuous map}\}$, where $\mathcal{S}^1$ is the unit sphere in $\R^2$. Then the second eigenvalue $\la_2$ of \eqref{problem} is characterized by 
\begin{equation}\label{Variational_second}
    \la_2=\inf\limits_{\A\in \E}\sup\limits_{u\in \A}J(u).
\end{equation}
Also, it is well known that an eigenfunction associated to $\la_2$ changes sign in $\Om$; cf. \cite{Lindqvist2006}. Let $u_2$ be a second eigenfunction of \eqref{problem}. The nodal set $\Nn_{u_2}$ of $u_2$ is defined as
$$ \Nn_{u_2}=\overline{\{x\in \Om: u_2(x)=0\}}.$$

For $p=2$, the structure of $\Nn_{u_2}$ is widely investigated after Payne conjectured \cite[Conjecture 5]{Payne1967} that $\Nn_{u_2}$ can not be closed. This conjecture has been proved under various symmetry and convexity assumptions on the domain; see, for instance, \cite{Lin1987,Putter,Jerison,Alessandrini} and references therein. In general, Payne's conjecture is not true; see \cite{Kennedy2013,Nadirashvili} for the counterexamples. In \cite{Anoop2016Drabek}, for $\Om$ a ball in $\R^N$ and for $p$-Laplacian, Anoop, Dr\'{a}bek, and Sasi proved that $\Nn_{u_2}$ can not be a concentric sphere. In the same article, the authors proposed an open problem \cite[Remark 4.2]{Anoop2016Drabek} regarding the validity of the Payne's conjecture for the $p$-Laplacian on a ball in $\R^N$. Later in \cite{Anoop2018}, the authors obtained similar results as in \cite{Anoop2016Drabek} for bounded radial domains in $\R^N$ (i.e., when $\Om$ is a ball or concentric annular region) using a different technique that relies on the strict monotonic behaviour of the first  eigenvalue $\la_1$ in annular regions. In \cite{Bobkov2019Kolonotskii}, Bobkov and Kolonitskii proved the conjecture \cite[Remark 4.2]{Anoop2016Drabek} even for more general domains in $\R^N$ having Steiner symmetry. However, in general, for doubly connected domains, the nature of the set $\Nn_{u_2}$ remains quite open even for the case $p=2$. In \cite{Sarswat2014nodal}, for $p=2$, the authors showed that $\Nn_{u_2}$ touches the boundary for doubly connected domains, which satisfy certain dihedral symmetry. In \cite[Theorem 1.4]{Kiwan2018}, for $p=2$ and for planar doubly connected domains with some convexity and symmetry assumptions, Kiwan proved that $\Nn_{u_2}$ can not enclose the inner hole. As an application of Theorem \ref{theorem}, we obtain an analogous result as in \cite[Theorem 1.4]{Kiwan2018} for the $p$-Laplacian, provided the domain has certain dihedral symmetry. Before giving the precise statement, we introduce a definition.

Let $P$ be a domain satisfying (\nameref{itm:A0}), (\nameref{itm:A1}), and (\nameref{itm:A2}). Let $\ga$ be a simple closed $C^2$ curve enclosing $P$. Let $Q$ denote the domain such that $\pa Q=\ga$. Assume that $Q$ also satisfies (\nameref{itm:A1}) and (\nameref{itm:A2}) with the same order of symmetry as that of $P$.

\begin{definition}\label{On_Off_2} For $P,\ga,Q$ as given above, we say that $P$ is in an ON position with respect to $\ga$ if an \textit{outer vertex} of $P$ lie on a half-axis of symmetry containing an \textit{outer vertex} of $Q$. Similarly, we say that $P$ is in an OFF position with respect to $\ga$ if an \textit{outer vertex} of $P$ lie on a half-axis of symmetry containing an \textit{inner vertex} of $Q$. 
\end{definition}

 Now we state our result.
\begin{theorem}\label{Nodal_theorem}
Let $P$ satisfy (\nameref{itm:A1}) for some fixed $n_0\in \N$, $n_0\geq 2$. Consider $\Om=B\setminus \overline{P}$ as in (\nameref{itm:Domain}), where $B$ is centered at $(-c,0)$ for some $c\geq 0$ and the center $o$ of $P$ is the origin $(0,0)$. Assume that $\frac{3}{2}<p<\infty$ and $P$ is not in an OFF position with respect to $B$. Suppose that  the nodal set $\Nn_{u_2}$ of a second eigenfunction $u_2$ of \eqref{problem} is a  $C^2$ curve satisfying (\nameref{itm:A1}) (for $n_0$), (\nameref{itm:A2}) and is concentric with $P$. Then, $\Nn_{u_2}$ can not have $P$ in an ON position with respect to it.
\end{theorem}

 The rest of the article is organized as follows. In Section \ref{preliminaries}, we recall some basic facts about the geometry of simply connected planar domains, and we mention a few auxiliary results, which is essential for the rest of the article. The sector reflection technique for the odd case is introduced in Section \ref{Mainresult}. We present the proof of Theorem \ref{theorem} and Theorem \ref{Nodal_theorem} in this section.
\section{Preliminaries}\label{preliminaries}
In this section, first, we introduce some notations which will be used later. Next, we mention a few elementary facts concerning the representation of normal to a planar domain on its boundary. Then we recall the shape derivative formula and some properties of the first eigenvalue of \eqref{problem}. Finally, we list a few comparison principles which will be needed in later sections. 

Throughout this article, we assume $n\geq 2$, $n \in \N$. Let $z_\alpha := \{ r e^{i \alpha}: r \in \mathbb{R}\}$, for $\alpha \in [0, 2 \pi]$, denote the line in $\mathbb{R}^2$ passing through origin $o$ making an angle $\alpha$ with the positive $x_1$-axis. Let $R_{\alpha}:\mathbb{R}^2 \rightarrow \mathbb{R}^2$ denote the reflection about the line $z_{\alpha}$. Observe that $z_\alpha = z_{\alpha+ \pi}$ for each $\alpha \in [0, 2\pi]$, where the addition is taken modulo $2\pi$. For a fixed $t \in \mathbb{R}$ and $a, b \in \mathbb{Z}$, $a< b$, we denote by  $\sigma( {t+\frac{a\pi}{n},t+\frac{b\pi}{n}})$, the sector delimited by the lines $z_{t+\frac{a\pi}{n}}$ and $z_{t+\frac{b\pi}{n}}$. More precisely,
\begin{equation}\label{sector}
     \sigma\left( {t+\frac{a\pi}{n},t+\frac{b\pi}{n}}\right):= \bigg\{ r e^{i \phi} \in  \mathbb{R}^2 : \phi \in \left(t+\frac{a \pi}{n},t+\frac{b\pi}{n}\right),\; r \in \mathbb{R} \bigg\}.
\end{equation}
For simplicity, we will write $\sigma_{(a,b)}$ to denote $\sigma\left({t+\frac{a\pi}{n},t+\frac{b\pi}{n}}\right)$ when $t$ is fixed. Next, recall that $\Om$ is as in (\nameref{itm:Domain}). Let $\Om_t$ be as mentioned in \eqref{t_configuration}. Now for $t\in[0,2\pi)$ and $k=0,1,\dots,2n-1$, we consider the following sectors of $\Om_t$:
\begin{equation}\label{Sectors}
    H_k(t):= \Om_t \cap \sigma_{(k,k+1)}\;;\;\widetilde{H}_k(t):= \overline{\Om_t} \cap \sigma_{(k,k+1)}.
\end{equation}

 For $\al\in[0,2\pi)$, let $S_\al$ denotes the half space containing $e_1=(1,0)$ such that $\pa S_\al=z_\al$. Then we define the following subdomains of $\Om_t$:
\begin{equation}\label{Half_space}
    O_\al^+(t)=S_\al \cap \Om_t\;;\;O_\al^-(t)=(\mathbb{R}^2\setminus S_\al) \cap \Om_t.
\end{equation}
We observe that $O_\al(t)$ (See Figure \ref{fig:Reflection}-(A)) is nothing but the smaller component of the two unequal parts of $\Om_t$ divided by $z_\al$.

Let $P$ be parametrized in polar coordinates in the following way:
\begin{equation}\label{P_parametrization}
P=\{ re^{\textbf{i}\phi}\, :\, \phi\in [0,2\pi), 0 \leq r < f(\phi)\}
\end{equation}
where $f: [0,2 \pi] \rightarrow [0, \infty)$ is a $\mathcal{C}^2$ map with $f(0)=f(2 \pi)$.

\subsection{Parametrization of $B$ and $P$:}\label{parametrization}
First, we recall a monotonicity property of the boundary of a disk, and then we mention a formula for the representation of the normal to a planar simply connected domain on its boundary.\par 
Let $B$ be centered at $(-c,0),\;c>0$. Let $r_1\;(r_1 >c)$ denote the radius of $B$. We consider the following parametrization of $B$ in polar coordinates with respect to $o=(0,0)$. 
\begin{equation}\label{disc_parametrization}
    B=\{re^{i\phi}: \phi\in[0,2\pi), 0\leq r < g(\phi)\},
\end{equation}
 where $g: [0,2 \pi] \rightarrow [0, \infty)$ is a $\mathcal{C}^2$ map with $g(0)= g(2 \pi)$ and the polar coordinates $(r,\phi)$ are measured with respect to the origin $o$ and the positive $x_1$-axis of $\mathbb{R}^2$. Then the following monotonicity property \cite[Lemma 4.1]{Anisa2020} holds on the boundary of the disk $B$.
 \begin{lemma}\label{disc_monotonicity}
 Let $\delta(\phi)$ denote the distance of a point $g(\phi) \, e^{i \phi}$ on $\partial B$ from $o=(0,0)$. Then  $\phi \mapsto \delta(\phi)$ is a strictly increasing function of $\phi$ in $[0,\pi]$ and is a strictly decreasing function of $\phi$ in $[\pi,2\pi]$.
 \end{lemma}
\par
Let $D\subset\mathbb{R}^2$ be a simply connected bounded domain which is given by
\begin{equation*}
    D=\{re^{i\phi}: \phi\in[0,2\pi), 0\leq r < h(\phi)\},
\end{equation*}
where $h$ is bounded and $2\pi$-periodic function of class $\mathcal{C}^2$. Let $v \in \mathcal{C}_0^\infty(\mathbb{R}^2)$ be a vector field that rotates $D$ by a right angle about the origin in the anticlockwise direction. Thus restriction of $v$ to $\partial D$ is given by 
$v(x_1,x_2) = (-x_2,x_1) \; \forall\; (x_1,x_2) \in \pa D$. This gives
$v \left( h(\phi) \left(\cos \phi, \sin \phi \right) \right)= h(\phi) \left(-\sin \phi, \cos \phi \right) \; \forall \;\phi \in [0, 2 \pi).$ Now viewing $\mathbb{R}^2$ as the complex plane $\mathbb{C}$, we can represent $v$ as $v(\zeta) = \textbf{i}\zeta \;\text{for all}\; \zeta = h(\phi)e^{\textbf{i}\phi} \in \partial D,$ which is equivalent to saying that $v(\phi):= v\left( h(\phi)e^{\textbf{i}\phi}\right)= \textbf{i} h(\phi)  e^{\textbf{i} \phi} \; \text{for all}\; \phi \in \mathbb{R}.$ \par
 Let $\eta$ denotes the unit outward normal to $D$ on $\pa D$. Now we state a lemma \cite[Lemma 4.2]{Anisa2020} without proof, which will be crucial for the proof of our main results.
\begin{lemma}\label{normal_presentation}
Let $D, h, v, \eta$, and $z_\alpha$ be as defined above. Then at any point $ h(\phi)e^{\textbf{i}\phi}$ of $\partial K$, we have the following:
\begin{enumerate}[(i)]
\item $\eta(\phi) := \eta(h(\phi)e^{\textbf{i}\phi}) = \dfrac{h(\phi)e^{\textbf{i}\phi}-\textbf{i}h^\prime(\phi)e^{\textbf{i}\phi}}{\sqrt{h^2(\phi)+(h^\prime(\phi))^2}} \; \forall \phi \in \mathbb{R}$.
\item $\left<\eta , v\right>(\phi) := \left<\eta , v\right> (h(\phi)e^{\textbf{i}\phi})  = \dfrac{-h(\phi)h^\prime(\phi)}{\sqrt{h^2(\phi)+(h^\prime(\phi))^2}} \; \forall \phi \in \mathbb{R}$. Thus $\left<\eta , v\right>$ has a constant sign on an interval $I \subset \mathbb{R}$ iff $h$ is monotonic in $I$.
\item If for some $\alpha \in [0, 2 \pi)$, the domain $D$ is symmetric with respect to the axis $z_{\alpha}$ then, for each $\theta \in [0, \pi]$,
$
\left<\eta , v\right>(\alpha+\theta) = -\left<\eta , v\right>(\alpha-\theta).
$
\end{enumerate}
\end{lemma}
 
\subsection{Initial position:}\label{Initial_position} Let $\Om$ be as in (\nameref{itm:Domain}). Here we fix the initial configuration, that is, the configuration at $t=0$. Let $P_t$ be in an OFF position with respect to $B$ at $t=0$. Then using parametrization \eqref{P_parametrization} of $P$, we observe that $f$ is an increasing function of $\phi$ on $({0,\frac{\pi}{n}})$ for $n$ even and is a decreasing function of $\phi$ on $({0,\frac{\pi}{n}})$ for $n$ odd. Also for $t\in [0,2\pi)$, the boundary of $P_t$ is given by,
\begin{equation*}
    \pa P_t:= \{f(\phi-t) e^{i \phi}: \phi \in [0, 2\pi)\}.
\end{equation*}
\subsection{Sector reflections for $n$ even:} The following sector reflection technique  is used in \cite[Lemma 6.1]{Anisa2020} for $P$ satisfying dihedral symmetry when $n$ is even.
\begin{lemma}\label{Sector_reflection_even}
Let $n\geq 2$ be even. Let $R_\al$ denote the reflection about the line $z_\al$. Then
\begin{enumerate}[(i)]
    \item \label{Containment_1_even} for  $k =0,2,4,\dots,n-2$, $$R_{t+\frac{(k+1)\pi}{n}}(H_k(t))\subsetneq H_{k+1}(t) \;\text{and}\; R_{t+\frac{(k+1)\pi}{n}}(\widetilde{H}_k(t)) \subsetneq \widetilde{H}_{k+1}(t) \setminus \partial B.$$
    \item \label{Containment_2_even} for $k=n, n+2,\dots, 2n-2$, $$R_{t+\frac{(k+1)\pi}{n}}(H_{k+1}(t)) \subsetneq H_{k}(t)\;\text{and}\; R_{t+\frac{(k+1)\pi}{n}}(\wide{H}_{k+1}(t)) \subsetneq \wide{H}_{k}(t) \setminus \pa B.$$
\end{enumerate}
\end{lemma}
\subsection{Shape derivative formula:}\label{Hadamard} Let $\Om_t$ be as in \eqref{t_configuration}. We set $\la_1(t):=\la_1(\Om_t)$, where $\la_1(\Om_t)$ is the first eigenvalue of \eqref{problem} on $\Om_t$. Also let $u_t$ be the positive eigenfunction associated to $\la_1(t)$ such that $||u_t||_p=1$. Let us choose a vector field $v \in \mathcal{C}_0^\infty(\Omega_t)$ defined as
\begin{equation}\label{vector_field}
v(\zeta) = \kappa(\zeta) \, \textbf{i}\zeta\; \text{for all} \; \zeta \in  \mathbb{C} \cong \mathbb{R}^2,
\end{equation}
where $\kappa:\mathbb{R}^2 \rightarrow [0,1]$ is a smooth function with compact support in $B$ such that $\kappa \equiv 1$ in a neighborhood of $P_t$ for all $t\in [0,2\pi)$. Then by the Hadamard perturbation formula (cf. \cite{Anisa2015,Melian2001}), we see that the map $t \mapsto \lambda_1(t)$ is a differentiable function, and its derivative is given by
\begin{equation}\label{hadamard}
    \la_1^\prime(t) =- (p-1)\int\limits_{x \in \partial P_t} \left|{\dfrac{\partial u_t}{\partial \eta_t}}(x)\right|^p \left<\eta_t, v\right>(x)\dS,
\end{equation}
where $\dS$ is the line element on $\partial P_t$ and $\eta_t(x)$ is the unit outward normal vector to $\Omega_t$ at $x \in \partial\Omega_t$.
 
\subsection{Properties of the first eigenvalue:}\label{Properties_lambda} 
Here, we list some properties of the map $t \longmapsto \la_1(t)$.
\begin{enumerate}[(i)]
    \item\label{Periodic Map} (Periodicity) Clearly, $P_{t+ 2 \pi} =P_t$, and hence $\Om_{t+2 \pi} =\Om_t$ for $t \in \R$. By the dihedral symmetry of $P$, we have $P_{t+  \frac{2 \pi}{n}} =P_t$, and hence $\Om_{t+\frac{2 \pi}{n}} =\Om_t$, not just for $t \in [0,2\pi)$, but also for $t \in \R$. Thus $\la_1(t+\frac{2\pi}{n})=\la_1(t)\;\forall\;t\in \R$. Therefore, the map $t\lmap\la_1(t)$ on $\R$ 
    is a periodic function with period $\frac{2\pi}{n}$.
    \item \label{Evenness} (Even Function) Since reflection about an axis of symmetry of $P$ is an isometry of $P$, $P_{ \frac{k\,\pi}{n}+t} =P_{ \frac{k\, \pi}{n}-t}$, $k \in \mathbb{Z}$. Hence, $\Om_{ \frac{k\,\pi}{n}+t} =\Om_{ \frac{k\, \pi}{n}-t}$, $k \in \mathbb{Z}$. In particular, $\Omega_{-t}=\Omega_t$ for all $t \in \mathbb{R}$. As a result, $\la_1(-t)=\la_1(t)\;\forall\;t\in \R$. Thus, $t\lmap\la_1(t)$ is an even function.
    \item \label{Critical_points} (Critical Points) By similar arguments as in the proof of \cite[Proposition 6.1]{Anisa2020} and using (i),(ii) given above, it follows that $\frac{k\pi}{n}$, $k=0,1,2,\dots,2n-1$, are the critical points of the map $t\lmap\la_1(t)$. In other words, $\la_1'(\frac{k\pi}{n})=0 \;\forall\;k=0,1,2,\dots,2n-1$.
\end{enumerate}
By the first two properties, it is enough to study the behaviour of the map $t\lmap\la_1(t)$ on the interval $[0,\frac{\pi}{n})$.

\subsection{Auxiliary results:} We recall the following strong maximum principle; cf. \cite[Theorem 7]{Protter}.
\begin{proposition}[strong maximum principle]\label{Strong_maximum}
Let $\Om\subset\mathbb{R}^N,\;N\geq 2$ be a domain and $u\in C^1(\overline{\Om})$ satisfies 
$$L(u):= \displaystyle\sum_{i,j=1}^{n}\frac{\pa}{\pa x_i}\bigg(a_{ij}\frac{\pa u}{\pa x_j}\bigg)\geq 0\;\;\text{in}\;\Om,$$
where $L$ is a uniformly elliptic operator on $\Om$ and $a_{ij}\in W_{\rm loc}^{1,\infty}(\Om)$. Suppose that $u\leq M$ in $\Om$ and  $u(x_0)=M$, where $x_0\in \pa \Om$ such that the interior sphere condition is satisfied at $x_0$. Then
$$\frac{\pa u}{\pa \eta}(x_0)>0,\;\;\text{unless}\; u\equiv M\;\text{in}\;\Om,$$
where $\eta$ denotes the unit outward normal to $\pa \Om$.
\end{proposition}
Next, we state a strong comparison principle for $p$-Laplacian in a weaker form, which is proved for a more general setting in \cite[Theorem 1.4]{Sciunzi2014}.
\begin{proposition}[Strong comparison principle]\label{Strong_comparison}
Let $\Om\subset\mathbb{R}^N,\;N\geq 2$ be a domain and $u,v\in C^1(\overline{\Om})$ satisfying 
$$-\Delta_p u-g(u)\leq -\Delta_p v-g(v)\;\;\text{in}\;\Om,$$
with $0\in\Om$ and $g$ be a locally Lipschitz function such that $g(x)>0$ for $x>0$. Assume that either $u$ or $v$ is non-negative solution of $-\Delta_p w=g(w)$ with $\frac{2N+2}{N+2}<p<\infty$. Then if $u\leq v$ in a connected subdomain $\Om'\subset \Om$, it follows that
$$u<v \;\;\text{in}\;\Om',\;\;\text{unless}\; u\equiv v \;\text{in} \;\Om'.$$
\end{proposition}

The following version of the weak comparison principle follows from \cite[Theorem 2.1]{Anisa2015Toledo}.
\begin{proposition}[Weak comparison principle]\label{Weak_comparison}
Let $\Om\subset\mathbb{R}^N,\;N\geq 2$ be a Lipschitz domain. Suppose that $1<p<\infty$ and $u,v\in C^1(\overline{\Om})$ be two non-negative weak solutions of $-\Delta_p w=\la w^{p-1}$ on $\Om$. Then if $u\leq v$ on $\pa \Om$,
$$u\leq v\;\text{on}\; \Om.$$
Furthermore if $x_0\in \pa \Om$ be such that $u(x_0)=0=v(x_0)$, then $\frac{\pa v}{\pa \eta}(x_0)\leq \frac{\pa u}{\pa \eta}(x_0)$.
\end{proposition}

\section{main results}\label{Mainresult}
\subsection{Sector reflections for $n$ odd}
In this section, we introduce a modified sector reflection technique for $P$ having dihedral symmetry when $n$ is odd by pairing sectors across the upper and lower hemispheres. To be precise, we have the following lemma.

\begin{figure}[t]\centering
\subfloat[
]{
\label{reflection_1}
\includegraphics[width=0.25\textwidth]{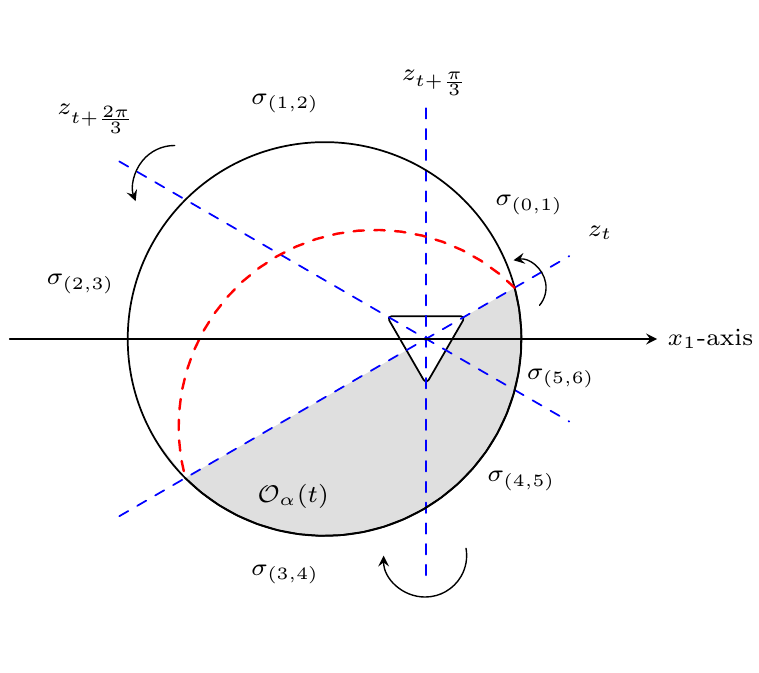}}
\hspace{20mm}
\subfloat[
]{
\label{reflection_2}
\includegraphics[width=0.25\textwidth]{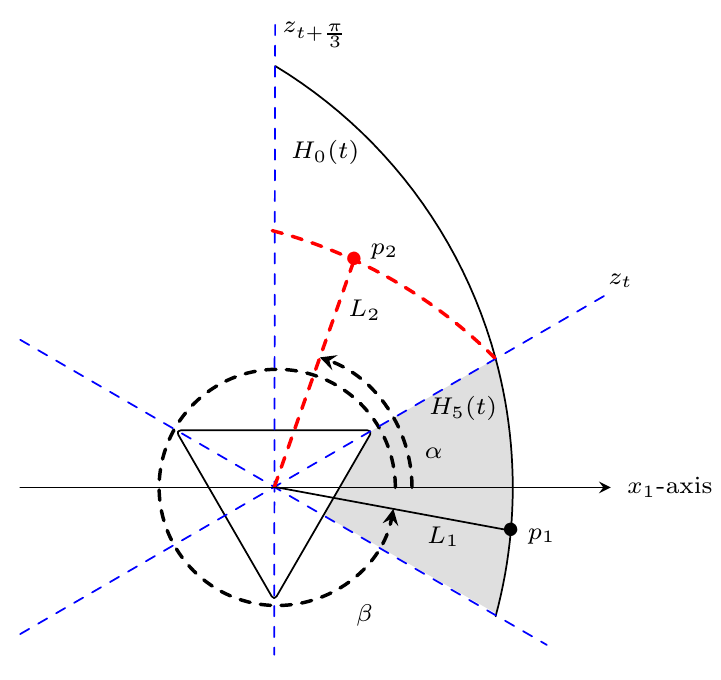}}
\caption{ (A) Sector pairings for $n$ odd. (B) Containment of sectors which lie on either side of the $x_1$-axis as in Lemma \ref{Sector_reflection_odd}. }\label{fig:Reflection}
\end{figure}

\begin{lemma}\label{Sector_reflection_odd}
Assume that $P$ satisfies (\nameref{itm:A1}) for some $n\in \N$, where $n \geq2$ is odd. Let $\Om_t$, for $t\in[0,2\pi)$, be as in \eqref{t_configuration} and $H_k(t),\widetilde{H}_k(t)$ be as defined in \eqref{Sectors}. Also let $R_\al$ denote the reflection about the axis $z_\al$. Then the following holds:
\begin{enumerate}[(i)]
    \item For  $k =2n-1,1,3,\dots,n-2$, $$R_{t+\frac{(k+1)\pi}{n}}(H_k(t))\subsetneq H_{k+1}(t) \;\text{and}\; R_{t+\frac{(k+1)\pi}{n}}(\widetilde{H}_k(t)) \subsetneq \widetilde{H}_{k+1}(t) \setminus \partial B.$$
    \item \label{Containment_2_odd} For $k=n, n+2,\dots, 2n-3$, $$R_{t+\frac{(k+1)\pi}{n}}(H_{k+1}(t)) \subsetneq H_{k}(t)\;\text{and}\; R_{t+\frac{(k+1)\pi}{n}}(\wide{H}_{k+1}(t)) \subsetneq \wide{H}_{k}(t) \setminus \pa B.$$
    \end{enumerate}
\end{lemma}
\begin{proof}
Let $n\in \mathbb{N}, n \geq 2$ be odd. Here we note that the sectors $\sigma_{(k, k+1)}$ and $\sigma_{( k+1, k+2)}$ both are completely above the $x_1$-axis for $k=1,3,\dots,n-4$ (Figure \ref{fig:Reflection}-(B)). Also the sectors $\sigma_{(k, k+1)}$ and $\sigma_{( k+1, k+2)}$ are both completely below the $x_1$-axis for $k=n,n+2,\dots,2n-3$. Thus, as in the proof of \cite[Lemma 6.1]{Anisa2020} (which was proved for $n$ even), we can prove $(i)$ and $(ii)$ for $k=1,3,\dots,n-4$  and $k=n,n+2,\dots,2n-3$. We omit the details here. Thus the inclusions only remain to be proved for $k =2n-1, n-2$. Let us consider the sectors $\sigma_{(2n-1, 2n)}$ and $\sigma_{( 0, 1)}$. We claim the following:
	\begin{equation}\label{claim_containment}
	    R_t (H_{2n-1}(t)) \subsetneq H_0(t) \;\text{ and }\;R_t(\widetilde{H}_{2n-1}(t)) \subsetneq \widetilde{H}_{0}(t) \setminus \partial B.
	\end{equation}
	Now it is enough to show that reflection of any point on $\partial H_{2n-1}(t) \cap \partial B$ lies completely inside the sector $H_{0}(t)$. 
	Let $p_1$ be a point on  $\partial H_{2n-1}(t) \cap \partial B$, and $L_1$ be the line joining $p_1$ to the center $o$ of $P$. Also, let $p_2$ be the reflection of the point $p_1$ about the $ z_t$-axis. Let $L_2$ be the line joining $p_2$ to $o$ (See Figure \ref{fig:Reflection}-(B)). Now, if $p_1$ lies above the $x_1$-axis, we are done by monotonicity of the map $g$ (Lemma \ref{disc_monotonicity}). Therefore, assume that $p_1$ is below the $x_1$-axis. Let $L_2$ make an angle $\theta$ and $L_1$ make an angle $\beta$ with the positive $x_1$-axis. Then $\theta=\alpha+t$, where $\alpha $ is the angle between $L_1$ and the $z_t$-axis.
	  We also note that 
	 $$\alpha=(2\pi-\beta)+t\implies\beta=2t+2\pi - \theta.$$
	  Since $t>0$, we have $2\pi-\theta < \beta $.
	 Moreover, $\beta\in [\pi,2\pi]$ and $2\pi - \theta \in [\pi,2\pi]$. For, if $2\pi-\theta < \pi$ then $\theta \geq \pi$, a contradiction. Therefore, as $g$ is strictly decreasing (Lemma \ref{disc_monotonicity}) on $[\pi,2\pi]$, we have $g(2\pi-\theta)> g(\beta)$. That is, $g(\theta)> g(\beta)$, as $g$ is $2\pi$ periodic. Therefore the length of $L_2$ is strictly bigger than the length of $L_1$, and hence $p_2$ must be inside of $H_{0}(t)$. Thus the  claim is proved.
	 
\noi	Now we consider the sectors $\sigma_{(n-2, n-1)}$ and $\sigma_{(n-1, n)}$ and proceed in the same way. Our claim is, 
	\begin{equation}
	    R_t (H_{n-2}(t)) \subsetneq H_{n-1}(t)\;  \text{and} \;R_t(\wide{H}_{n-2}(t)) \subsetneq \wide{H}_{n-1}(t) \setminus \partial B.
	\end{equation}
	
\noi	As before, let $p_1^\prime$ be a point in $\partial H_{n-2}(t) \cap \partial B$, and $p_2^\prime$ be its reflection about the $z_{t+\frac{n-1}{n}\pi}$-axis. Let $L_1^\prime$ and $L_2^\prime$ be the lines joining $o$ to $p_1^\prime$ and $p_2^\prime$, respectively. 
Assume that $p_2^\prime $ lies below the $x_1$-axis. Let $\beta^\prime$ (resp. $\theta^\prime$) be the angle made by $L_1^\prime$ (resp. $L_2^\prime$) with the positive $x_1$-axis.
	Then clearly, 
	\begin{align*}
	    \theta^\prime &= 2(t+ \frac{n-1}{n}\pi) - \beta^\prime.
	    \end{align*}
	    This implies that
	\begin{align*} 2\pi- \theta^\prime &= 2\pi- 2(t+\frac{n-1}{n}\pi)+\beta^\prime= 2(\frac{\pi}{n}-t)+\beta^\prime > \beta^\prime,\;\;\text{as} \;t\in (0,\frac{\pi}{n}).
	\end{align*}
	Also, $\beta^\prime \in [0,\pi]$ and $2\pi-\theta^\prime \in [0,\pi]$. Otherwise, $\theta^\prime < \pi$, which is not possible. Since $g$ is a strictly increasing function on $[0,\pi] $  (Lemma \ref{disc_monotonicity}), we have $g(2\pi-\theta^\prime)> g(\beta^\prime)$. Thus $g(\theta^\prime )> g(\beta^\prime)$ by the periodicity of $g$ and hence $p_2^\prime $ lies inside $ H_{n-1}(t)$. This finishes the proof.
\end{proof}

 \subsection{Monotonicity of the eigenfunctions} 
  Here we prove a few monotonicity properties of the first eigenfunctions of $\eqref{problem}$ on the sectors formed by the consecutive axes of symmetries of $P$. We also establish the monotonicity of the normal derivatives of the first eigenfunctions on these sectors. First, we introduce some notations. 
  
  Let $t\in [0,2\pi)$.
\begin{align}
    \text{For $n$ odd},\; H_{\rm o}(t)&=\displaystyle\Bigg(\bigcup_{\substack{2n-1\leq k\leq n-2\\k \;\rm{odd }}} H_k(t)\Bigg)\cup \Bigg(\bigcup_{\substack{n\leq k\leq 2n-3\\k \;\rm{odd}}} H_{k+1}(t)\Bigg)\label{Odd_sectors},\\
    \Ga_{\rm o}(t)&=\displaystyle\Bigg(\bigcup_{\substack{2n-1\leq k\leq n-2\\k \;\rm{odd }}} (\pa P_t\cap \sigma_{(k,k+1)})\Bigg)\cup \Bigg(\bigcup_{\substack{n\leq k\leq 2n-3\\k \;\rm{odd}}} (\pa P_t\cap \sigma_{(k+1,k+2)})\Bigg)\subset  \pa H_{\rm o}(t).\nonumber
    \end{align}
    \begin{align}
    \text{For $n$ even},\;H_{\rm e}(t)&=\displaystyle\Bigg(\bigcup_{\substack{0\leq k\leq n-2\\k \;\rm{even}}} H_k(t)\Bigg)\cup \Bigg(\bigcup_{\substack{n\leq k\leq 2n-2\\k \;\rm{even}}} H_{k+1}(t)\Bigg)\label{Even_sectors},\\
    \Ga_{\rm e}(t) &=\displaystyle\Bigg(\bigcup_{\substack{0\leq k\leq n-2\\k \;\rm{even }}} (\pa P_t\cap \sigma_{(k,k+1)})\Bigg)\cup \Bigg(\bigcup_{\substack{n\leq k\leq 2n-2\\k \;\rm{even}}} (\pa P_t\cap \sigma_{(k+1,k+2)})\Bigg)\subset  \pa H_{\rm e}(t).\nonumber
\end{align}
\begin{proposition}\label{Hopf_lemma} 
Let $n\in \mathbb{N},\;n\geq 2$ and $1<p<\infty$.  Consider \eqref{problem} defined on $\Omega_t$, $t \in [0, 2 \pi)$. Let $u_t$ be a positive eigenfunction associated to $\la_1(t)$. Let $H_{\rm o}(t),\;H_{\rm e}(t),\;\Ga_{\rm o}(t)$ and $\Ga_{\rm e}(t)$ be as defined in \eqref{Odd_sectors} and \eqref{Even_sectors}.
\begin{enumerate}[(i)]
    \item 
    For $n$ odd, let $v_t$ be defined in $H_{\rm o}(t)$ as follows 
    $$v_t(x):=u_t\big(R_{t+\frac{(k+1)\pi}{n}}(x)\big) \;\text{for}\;\begin{cases}
    x \in H_k(t),\;\;
    k\in \{2n-1,1,\dots,n-2\},\\
      x \in H_{k+1}(t),\;\;
      k\in \{n,n+2,\dots,2n-3\}.
    \end{cases}$$
    Then $u_t\leq v_t$ in $H_{\rm o}(t)$ and $\frac{\pa v_t}{\pa \eta_t}\leq\frac{\pa u_t}{\pa \eta_t}$ on $\Ga_{\rm o}(t)\subset  \pa H_{\rm o}(t)$. Furthermore, both these inequalities are strict for $\frac{3}{2}<p<\infty$.
    
    \item \label{Neumann_derivative_even} For $n$ even, let $v_t$ be defined in $H_{\rm e}(t)$ as follows 
    $$v_t(x):=u_t\big(R_{t+\frac{(k+1)\pi}{n}}(x)\big) \;\text{ for }\;\begin{cases}
    x \in H_k(t),\;\;
    k\in \{0,2,\dots,n-2\},\\
        x \in H_{k+1}(t),\;\;
        k\in \{n,n+2,\dots,2n-2\}.
    \end{cases}$$
    Then $u_t\leq v_t$ in $H_{\rm e}(t)$ and $\frac{\pa v_t}{\pa \eta_t}\leq \frac{\pa u_t}{\pa \eta_t}$ on $\Ga_{\rm e}(t)\subset  \pa H_{\rm e}(t)$. Furthermore, both these inequalities are strict for $\frac{3}{2}<p<\infty$.
\end{enumerate}
\end{proposition}
\begin{proof}
We give proof for $(i)$, i.e., when $n$ is odd. The $n$ even case, i.e., part $(ii)$ of the Proposition, can be proved in a similar manner.

$(i)$ Let $n\in \mathbb{N}$, $n\geq 2$, be odd. First we prove the strict inequality case. Assume that $\frac{3}{2}<p<\infty$. Let $t\in [0,2\pi)$ be fixed. For convenience we set $u:=u_t,\;v:=v_t$ and $\eta:=\eta_t$. Let $k_0\in\{2n-1,1,\dots,n-2\}$ be fixed. Also, let $\al=t+\frac{(k_0+1)\pi}{n}$. Consider $O_\al^+(t)$ as in \eqref{Half_space}. Observing that $R_{\al}(O_\al^+(t))\subsetneq O_\al^-(t)$ (Figure \ref{fig:Reflection}-(A)), we define $\tilde{v}$ in $O_\al^+(t)$ as follows:
$$\tilde{v}(x):=u\big(R_{\al}(x)\big)\;\text{for}\;x\in O_\al^+(t).$$
Let $u_0$ denote the restriction of the function $u$ to $\pa O_\al^+(t) \cap z_{\al}$. Then $u$ and $\tilde{v}$ satisfy the following equations weakly:
\begin{center}
\noindent\begin{minipage}{.3\linewidth}
\begin{equation*}
	 \begin{aligned}
	-\Delta_p u &=\la_1 u^{p-1}, \\
	u &=0, \\
	u &=u_0, \\
	u&=0,
	\end{aligned}
\end{equation*}
\end{minipage}
\begin{minipage}{.3\linewidth}
{\begin{equation*}
	 \begin{aligned}
	-\Delta_p \tilde{v} &=\la_1 \tilde{v}^{p-1} \;\;\; \text{in} \quad O_\al^+(t),\\
	\tilde{v} &=0 \qquad\quad\; \text{on}\; \pa O_\al^+(t) \cap\pa P_t,\\
	 \tilde{v} &=u_0 \qquad\;\;\; \text{on}\; \pa O_\al^+(t) \cap z_{\al},\\
	 \tilde{v} &>0 \qquad\quad\; \text{on}\; \pa O_\al^+(t) \cap\pa B.\\
	\end{aligned}
	\end{equation*}
}
\end{minipage}
\end{center}
 \noi Note that $\tilde{v}\geq u$ on $\pa O_\al^+(t)$. Since $u\in C^1(\overline{\Om_t})$ (cf. \cite[Theorem 1.3]{Barles1988}), by the weak comparison principle (Proposition \ref{Weak_comparison}), we obtain $\tilde{v}\geq u$ in $O_\al^+(t)$. Therefore the strong comparison principle (Proposition \ref{Strong_comparison}) yields 
\begin{equation}\label{Strict_comparison_1}
    \tilde{v}>u \;\text{in}\; O_\al^+(t).
\end{equation}

\noi Let us consider the sector $H_{k_0}(t)$ and let $v$ be defined in $H_{k_0}(t)$ as in $(i)$, i.e., $v(x):=u\big(R_{\al}(x)\big)$ for $x\in H_{k_0}(t) $.
Now by Lemma \ref{Sector_reflection_odd}, $v$ is well defined in $H_{k_0}(t)$. Also $v=\tilde{v}$ on $z_{t+\frac{k_0\pi}{n}}\cap O_\al^+(t)$. Therefore by \eqref{Strict_comparison_1}, we have $v>u$ on $z_{t+\frac{k_0\pi}{n}}\cap O_\al^+(t)$. Thus $u$ and $v$ satisfy the following equations weakly:
\begin{center}
\begin{minipage}{.3\linewidth}
\begin{equation*}
	 \begin{aligned}
	-\Delta_p u &=\la_1 u^{p-1}, \\
	u &=0, \\
	u &=u_0,\\
	u&=0,
	\end{aligned}
\end{equation*}
\end{minipage}
\begin{minipage}{.5\linewidth}
{\begin{equation*}
	 \begin{aligned}
	-\Delta_p v&=\la_1 v^{p-1} \quad\text{in} \quad H_{k_0}(t),\\
	v &=0 \qquad \quad\;\;\text{on}\; \pa H_{k_0}(t)  \cap\pa P_t,\\
	 v &>u_0 \qquad\;\;\;\;\text{on}\; \big(\pa H_{k_0}(t) \cap z_{t+\frac{k_0\pi}{n}}\big)\setminus\big(\pa P_t\cup\pa B\big),\\
	 v &>0 \qquad \quad\;\;\text{on}\; \pa H_{k_0}(t) \cap\pa B.
	\end{aligned}
\end{equation*}}
\end{minipage}
\end{center}

\noi As earlier, using the weak comparison principle and the strong comparison principle, we get 
\begin{equation}\label{Strict_comparison_2}
    v>u \;\text{in}\; H_{k_0}(t).
\end{equation}
Since $k_0$ is arbitrary, this strict inequality holds true for any $k\in\{2n-1,1,\dots,n-2\}$. In a similar manner, we can show that $v>u$ in $H_{k+1}(t)$ for $k\in\{n,n+2,\dots,2n-3\}$. The first part of $(i)$ follows after combining all these cases.

\noi Let $w:=v-u$. Then $w$ weakly satisfies the following linearized differential equation in $H_{k_0}(t)$:
\begin{equation}
\begin{aligned}\label{Linearized_problem}
    -{\rm div}(A\nabla w)&=\la_1(v^{p-1}-u^{p-1})>0\;\text{in}\; H_{k_0}(t),\\
    w&\geq 0\;\qquad\qquad\qquad\qquad\;\text{on}\; \pa H_{k_0}(t),
\end{aligned}
\end{equation}
where $A(\xi) = |\xi|^{p-2} · \xi = D\Gamma(\xi)$ for $\xi = (\xi_1, \dots,\xi_n) \in \R^n$, $\Gamma : \R^n \rightarrow \R$ is a strictly convex function  defined by $\Gamma(\xi) = \frac{|\xi|^p}{p}$. Moreover, $\Gamma \in \C^\infty(\R^n \setminus\{0\})$ and $A(x)=[a_{ij}(x)]$ is given by
$$a_{ij}(x) =\int_0^1 D_{ij}\Gamma(t\nabla u(x) + (1 - t) \nabla v(x))\dt,$$ where $[D_{ij}\Gamma]$ is the Hessian matrix of $\Gamma$. Let $x_0\in \pa P_t\cap \pa H_{k_0}(t)$. Then by Proposition \ref{Strong_maximum}$, \frac{\pa u}{\pa \eta}(x_0)<0$. Now by similar arguments as in the proof of \cite[Theorem 3.2]{Anisa2015Toledo}, we find a neighbourhood $U$ of $x_0$ in $\overline{\Omega_t}$ such that $A(x)$ is uniformly positive definite on $U$. Thus there exists a neighbourhood $V$ of $\pa P_t\cap \pa H_{k_0}(t)$ such that $A(x)$ is uniformly positive definite and hence the operator in \eqref{Linearized_problem} is uniformly elliptic in $V\cap H_{k_0}(t)$. Now since, $w>0$ in $H_{k_0}(t)$ (using \eqref{Strict_comparison_2}) and $w=0$ on $\pa P_t\cap \pa H_{k_0}(t)$, the strong maximum principle (Proposition \ref{Strong_maximum}) implies that \begin{equation}\label{neumann1}
    \frac{\pa v}{\pa \eta}-\frac{\pa u}{\pa \eta}=\frac{\pa w}{\pa \eta}<0 \;\text{on}\;\pa P_t\cap \pa H_{k_0}(t).
\end{equation}
Since $k_0\in\{2n-1,1,\dots,n-2\}$ is arbitrary, \eqref{neumann1} holds on $\pa P_t\cap \pa H_{k}(t)$ for all
$k\in\{2n-1,1,\dots,n-2\}$. Proceeding in the same way, we can show that
\begin{equation*}
    \frac{\pa v}{\pa \eta}<\frac{\pa u}{\pa \eta}\;\text{on}\;\pa P_t\cap \pa H_{k+1}(t),
\end{equation*}
for $k\in \{n,n+2,\dots,2n-3\}$. Combining all, we conclude that $\frac{\pa v}{\pa \eta}<\frac{\pa u}{\pa \eta}\;\text{on}\;\Ga_{\rm o}(t)$. The proof for the case $1< p \leq\frac{3}{2}$ will follow along the same line of ideas by using the weak comparison principle (Proposition \ref{Weak_comparison}). This completes the proof.
\end{proof}

In the following proposition, we show that the map $t\longmapsto \la_1(t)$ is monotonic in the interval $(0,\frac{\pi}{n})$.  Here, we apply the monotonicity properties of the first eigenfunctions obtained in the previous proposition. 
\begin{proposition}\label{Monotonicity_lambda}
Let $n\in \mathbb{N}$, $n\geq 2$. Then the following holds for the map $t \longmapsto \lambda_1(t)$:
\begin{enumerate}[(i)]
    \item For $\frac{3}{2}<p<\infty$, $\la_1'(t)>0$ for all $t\in(0,\frac{\pi}{n})$.
    \item For $1<p\leq\frac{3}{2}$, $\la_1'(t)\geq 0$ for all $t\in(0,\frac{\pi}{n})$.
    \end{enumerate}
\end{proposition}
\begin{proof} Let $t\in(0,\frac{\pi}{n})$ be fixed, and $u_t$ be the positive eigenfunction of \eqref{problem} corresponding to $\la_1(t)$ such that $||u_t||_p$=1. Now we consider the two cases.\\
\underline{$(a)$ $n$ odd}:
We rewrite the shape derivative formula \eqref{hadamard} as follows:
\begin{align}
 \la_1^\prime(t)
    & = -(p-1)\sum_{k=0}^{2n-1} \int\limits_{\partial P_t \;\cap\; \sigma\left({k,k+1}\right)} \left|{\dfrac{\partial u_t}{\partial \eta_t}}(x)\right|^p\left< \eta_t, v\right>(x)\dS.\nonumber
\end{align}
Thus, 
\begin{align}\label{integral_1}
    \la_1^\prime(t) = -(p-1)\sum_{k=2n-1}^{n-1} \int\limits_{\pa P_t\; \cap\; \sigma_{(k,k+1)}} \left|{\dfrac{\pa u_t}{\pa \eta_t}}(x)\right|^p ~ \left< \eta_t,  v \right>(x)\dS\\
	 -(p-1)\sum_{k=n}^{2n-2} \int\limits_{\partial P_t\; \cap\; \sigma_{(k,k+1)}} \left|{\dfrac{\pa u_t}{\partial \eta_t}}(x)\right|^p \left< \eta_t, v \right>(x)\dS.\nonumber
\end{align}

Let $k\in\{2n-1,1,3,\dots,n-2\}$. Then by Lemma \ref{normal_presentation}, for $\phi \in [0 , \frac{\pi}{n}]$, we have 
	\begin{equation}
	\begin{aligned}
	\label{ref_0}\left<\beta_t , v \right> \left({t+\frac{(k_0+1)\pi}{n}+\phi}\right) = -\left< \beta_t,  v\right> \left({t+\frac{(k_0+1)\pi}{n}-\phi}\right),
	\end{aligned}
	\end{equation}
	where $\beta_t $ is the unit outward normal to $P_t$ on its boundary, i.e., $\beta_t = - \eta_t$. Now using the symmetry of $P_t$, it is easy to observe that $$R_{t+\frac{(k+1)\pi}{n}}(\pa P_t \cap \sigma_{(k, k+1)})=\pa P_t \cap \sigma_{(k+1, k+2)},$$
	for all $k\in\{2n-1,1,3,\dots,n-2\}.$ Thus if $x^\prime:= R_{t+\frac{(k+1)\pi}{n}}(x)$ for $x\in \pa P_t \cap \sigma_{(k,k+1)}$, then we have
	\begin{equation}\label{Eta_sign_1}
	  \left<\eta_t,  v\right> \left( x^\prime \right) = - \left< \eta_t , v \right>\left( x \right)\;\;\forall \; x \in  \pa P_t \cap \sigma_{(k,k+1)}.
	  \end{equation}
	 Therefore using \eqref{Eta_sign_1}, for each $k = 2n-1,1,3,\dots,n-2$, we get 
		\begin{equation}\label{Split_1}
		\begin{aligned}
		\int\limits_{\pa P_t \;\cap\; \sigma_{(k,k+1)}} \left|{\dfrac{\pa u_t}{\pa \eta_t}}(x)\right|^p~\left< \eta_t,  v\right>(x)\dS + \int\limits_{\pa P_t \;\cap\; \sigma_{(k+1,k+2)}} \left|{\dfrac{\pa u_t}{\pa \eta_t}}(x)\right|^p \left<\eta_t , v\right>(x)\dS\\
		=\int\limits_{\pa P_t\; \cap \;\sigma_{(k,k+1)}} \left({\left|{\dfrac{\pa u_t}{\pa \eta_t}}(x)\right|^p-\left|{\dfrac{\pa u_t}{\partial \eta_t}}(x^\prime)\right|^p}\right) \left<\eta_t,  v \right>(x)\dS.
		\end{aligned}
		\end{equation}
Proceeding in a similar manner, for $k = n,n+2, \dots, 2n-3$, we have
\begin{equation}\label{Split_2}
		\begin{aligned}
		\int\limits_{\pa P_t \;\cap\; \sigma_{(k,k+1)}} \left|{\dfrac{\pa u_t}{\pa \eta_t}}(x)\right|^p~\left< \eta_t,  v\right>(x)\dS + \int\limits_{\pa P_t \;\cap\; \sigma_{(k+1,k+2)}} \left|{\dfrac{\pa u_t}{\pa \eta_t}}(x)\right|^p \left<\eta_t , v\right>(x)\dS\\
		=\int\limits_{\pa P_t\; \cap \;\sigma_{(k+1,k+2)}} \left({\left|{\dfrac{\pa u_t}{\pa \eta_t}}(x)\right|^p-\left|{\dfrac{\pa u_t}{\partial \eta_t}}(x^\prime)\right|^p}\right) \left<\eta_t,  v \right>(x)\dS.
		\end{aligned}
		\end{equation}
Therefore, using \eqref{Split_1} and \eqref{Split_2}, expression \eqref{integral_1} can be rewritten as:
\begin{align}\label{integral_2}
    \la_1^\prime(t) = (p-1)\sum_{\substack{2n-1\leq k\leq n-2\\k \mbox{ \small{odd} }}} \int_{\pa P_t\; \cap\; \sigma_{(k,k+1)}} \left({\left|{\dfrac{\pa u_t}{\pa \eta_t}}(x')\right|^p-\left|{\dfrac{\pa u_t}{\partial \eta_t}}(x)\right|^p}\right)\left< \eta_t,  v \right>(x)\dS\\
	 +(p-1)\sum_{\substack{n\leq k\leq 2n-3\\k \mbox{ \small{odd} }}}\int_{\partial P_t\; \cap\; \sigma_{(k,k+1)}} \left({\left|{\dfrac{\pa u_t}{\pa \eta_t}}(x')\right|^p-\left|{\dfrac{\pa u_t}{\partial \eta_t}}(x)\right|^p}\right) \left< \eta_t, v \right>(x)\dS.\nonumber
\end{align}
Now from the parametrization \eqref{P_parametrization} of $P$, we have that $f$ is an increasing function of $\phi$ on $[0,\frac{\pi}{n}]$ (Section \ref{Initial_position}). Also, by noting $\beta_t=-\eta_t$ and applying Lemma \ref{normal_presentation} for $\beta_t$, we have 
\begin{equation}
    \label {Sign_normal_1} \left< \eta_t,  v \right> = - \left< \beta_t,  v \right>  >0 ~ \mbox{ on }\partial P_t \cap \sigma_{(k,k+1)}~\mbox{ for each } k = 2n-1,1,3, \dots , n-2.
\end{equation}

$(i)$ Let $\frac{3}{2}<p<\infty$ and let $v_t$ be the function as defined in Proposition \ref{Hopf_lemma}-$(i)$. Note that, for $k=2n-1,1,\dots,n-2$,
\begin{equation*}
    \frac{\pa v_t }{\pa \eta_t}(x)=\frac{\pa u_t}{\pa\eta_t}(x')\;\;\text{for}\;x\in H_k(t).
\end{equation*}
Since $\frac{\pa u_t}{\pa\eta_t}<0$ on $\pa P_t\cap \sigma_{(k,k+1)}$ for each $k=2n-1,1,\dots,n-2$, Proposition \ref{Hopf_lemma}-$(i)$ yields,
\begin{equation*}
    \frac{\pa v_t }{\pa \eta_t}<\frac{\pa u_t}{\pa\eta_t}<0\;\;\text{on}\;\pa P_t\cap \sigma_{(k,k+1)}\subset \Ga_0(t).
\end{equation*}
Thus for each $k=2n-1,1,\dots,n-2$, we have 
\begin{equation}\label{comparison_1}
    \left|\frac{\pa v_t }{\pa \eta_t}\right|^p>\left|\frac{\pa u_t}{\pa\eta_t}\right|^p\;\;\text{on}\;\pa P_t\cap \sigma_{(k,k+1)}.
\end{equation}
Therefore, equation \eqref{comparison_1} and \eqref{Sign_normal_1} gives that the first integral in \eqref{integral_2} is positive. In a similar way, we can show that the second integral in \eqref{integral_2} is positive too. Hence the  conclusion follows.

$(ii)$ For $1<p\leq \frac{3}{2}$, following the same arguments as in the previous case and using Proposition \ref{Hopf_lemma}-$(i)$, we have 

\begin{equation}\label{comparison_11}
    \left|\frac{\pa v_t }{\pa \eta_t}\right|^p\geq \left|\frac{\pa u_t}{\pa\eta_t}\right|^p\;\;\text{on}\;\pa P_t\cap \sigma_{(k,k+1)},
\end{equation}
for $k=2n-1,1,\dots,n-2$ and $k=n,n+2,\dots,2n-3$. Thus the assertion follows from \eqref{comparison_11}, \eqref{Sign_normal_1} and \eqref{integral_2}.

\noi\underline{$(b)$ $n$ even}: We rewrite the shape derivative formula \eqref{hadamard} as follows:
\begin{align}\label{integral_3}
    \la_1^\prime(t) = -(p-1)\sum_{k=0}^{n-1} \int\limits_{\pa P_t\; \cap\; \sigma_{(k,k+1)}} \left|{\dfrac{\pa u_t}{\pa \eta_t}}(x)\right|^p ~ \left< \eta_t,  v \right>(x)\dS\\
	 -(p-1)\sum_{k=n}^{2n-1} \int\limits_{\partial P_t\; \cap\; \sigma_{(k,k+1)}} \left|{\dfrac{\pa u_t}{\partial \eta_t}}(x)\right|^p \left< \eta_t, v \right>(x)\dS.\nonumber
\end{align}
Now using Lemma \ref{Sector_reflection_even} and following the same way as in the previous case, we can further rewrite the expression as,
\begin{align}
    \label{integral_4}
    \la_1^\prime(t) = (p-1)\sum_{\substack{0\leq k\leq n-2\\k \mbox{ \small{even} }}} \int_{\pa P_t\; \cap\; \sigma_{(k,k+1)}} \left({\left|{\dfrac{\pa u_t}{\pa \eta_t}}(x')\right|^p-\left|{\dfrac{\pa u_t}{\partial \eta_t}}(x)\right|^p}\right)\left< \eta_t,  v \right>(x)\dS\\
	 +(p-1)\sum_{\substack{n\leq k\leq 2n-2\\k \mbox{ \small{even} }}}\int_{\partial P_t\; \cap\; \sigma_{(k,k+1)}} \left({\left|{\dfrac{\pa u_t}{\pa \eta_t}}(x')\right|^p-\left|{\dfrac{\pa u_t}{\partial \eta_t}}(x)\right|^p}\right) \left< \eta_t, v \right>(x)\dS.\nonumber
\end{align}
Hence  the assertion follows along the same line of ideas using Proposition  \ref{Hopf_lemma}-$(ii)$ and Lemma \ref{normal_presentation}. This finishes the proof.
\end{proof}
\begin{proof}[Proof of Theorem \ref{theorem}.] 
The conclusions of the theorem follow immediately from Proposition \ref{Monotonicity_lambda} and from the properties of $\la_1$ listed in Section \ref{Properties_lambda}. 
\end{proof}

\begin{remark}\label{Final_remark}
\begin{enumerate}[(i)]
 \item Let $n\in \N$, $n\geq 2$ be fixed. Let $P_1$ and $P_2$ be two planar domains satisfying (\nameref{itm:A0}), (\nameref{itm:A1}) (for $n$) and (\nameref{itm:A2}) such that both $P_1$ and $P_2$ are centered at $(0,0)$. Suppose that $P_2\subsetneq P_1$. Now following Definition \ref{On_Off_2}, we call $P_2$ to be in an ON (OFF) position with respect to $P_1$ if $P_2$ is in an ON (OFF) position with respect to $\pa P_1$; cf. \cite{kiwanDihed}. Now adapting the ideas similar to the ones used so far, we can deduce results analogous to Theorem \ref{theorem} for $\Om=P_1\setminus \overline{P_2}$. More precisely, for $\frac{3}{2}<p<\infty$, among all rotations of $P_2$ within $P_1$, $\la_1$ is maximum (minimum) only when $P_2$ is in ON (OFF) position with respect to $P_1$. Similarly, for $1<p\leq\frac{3}{2}$, among all rotations of $P_2$ within $P_1$, $\la_1$ is maximum (minimum) when $P_2$ is in ON (OFF) position with respect to $P_1$.
 
 \item (Global  optimizers) Here we would like to comment on the global maximizer and minimizer with respect to all rigid motions, i.e., rotations and  the displacements of the obstacle $P$ along the $x_1$-axis. The global maximum, i.e., maximizer with respect to the
translations and rotations of $P$ within B, occurs when $B$ and $P$ are concentric. The proof of this fact follows from \cite[Theorem 5.2]{Anisa2020}. In the concentric position, since rotations of $P$ do not change $\la_1$ (as domains remain isometric), $\la_1(\Om_t)$ is the maximum for all $t\in[0,2\pi)$. Similarly, the global minimum occurs when $P$ is in an OFF position with respect to $B$ with a pair of its consecutive outer vertices touching $\pa B$; cf. \cite[Theorem 5.2]{Anisa2020}. 
 
 \item\label{non-smooth} For $p=2$, we can extend these results to concentric polygons having non-smooth boundaries. The difficulty arises because the first eigenfunction of (\ref{problem}) is not in 
 $\C^\infty(\overline{\Omega})$. However, it is in $H^{1+\delta}(\Omega)$ for some $\delta \in (\frac{1}{2}, \frac{3}{5})$. As a result, the shape calculus part of the proof, leading to the expression for $\lambda_1^\prime$, becomes more technical. We follow \cite{AithalRaut} for this extension.
 
 \item The extension mentioned in $(iii)$ can further be generalised to the same family of domains on certain non-Euclidean spaces as well. Following \cite{AithalRaut, AnisaAithal2005}, one can generalise these results to such family of domains in all the three space forms (viz., the Euclidean plane $\mathbb{E}^2$, the Riemann sphere $S^2$, and the hyperbolic space $\mathbb{H}^2$). For space forms, reflection about a geodesic is always an isometry of the Riemannian manifold. Hence, the reflection technique follows for the spherical as well as the hyperbolic polygons. 
 
 \item Consider the same family of non-smooth polygonal domains as described in $(iii)$. We now consider the following stationary problem $-\De u=1$
 in $\Omega$, $u=0$ on $\pa \Omega$. 
 Then, as in \cite{AithalSaraswat}, we can show that the associated Dirichlet energy functional $\int_\Omega \|u\|^2\, dx$ attains its extremum values when the axes of symmetry of the polygons coincide.
 
 \item From \cite{Anisa2021Souvik}, it follows that we can generalize these results to a generic second-order uniformly elliptic operator $L$ in the divergence form and consider the associated Dirichlet
eigenvalue problem. Here, 
\begin{equation*}
Lu\equiv -\nabla\cdot(a(x)\nabla u) = \lambda_1(t) u,~ u > 0 ~\mbox{in } \Omega_t, 
~ u(t) = 0~ \mbox{on } \partial\Omega_t, ~
\int_{\Omega_t} u^2(x)~dx = 1.
\end{equation*}
The coefficient $a(x)$ is chosen such that  
the operator $L$ is invariant under rotations and translations in the plane. 
 
 \end{enumerate}
\end{remark}

Now we give a proof of Theorem \ref{Nodal_theorem}. In the proof, we use the strict monotonicity of the first eigenvalue for $\frac{3}{2}<p<\infty$, obtained in Theorem \ref{theorem}, and the variational characterization \eqref{Variational_second} of the second eigenvalue. 

\begin{proof}[Proof of Theorem \ref{Nodal_theorem}]
Let the nodal set $\Nn_{u_2}$ of $u_2$ satisfy the properties mentioned in Theorem \ref{Nodal_theorem}. If possible, let $\Nn_{u_2}$ enclose $P$ in such a way that $\Nn_{u_2}$ is concentric with $P$ and $P$ is in an ON position with respect to it. Let $\Om_1\subset \Om$ be the domain enclosed by $\Nn_{u_2}$ and $\pa P$,  and $\Om_2=\Om\setminus \overline{\Om_1}$. Then $\Om_1$ and $\Om_2$ are the nodal domains of $u_2$. Without loss of generality we assume that $\Om_1=\{x\in \Om:u_2(x)>0\}$ and $\Om_2=\{x\in \Om:u_2(x)<0\}$. Since $u_2$ does not change sign in $\Om_1$ as well as in $\Om_2$, we have $\la_1(\Om_1)=\la_2=\la_1(\Om_2).$ By hypothesis, $P$ is in an ON position with respect to $\Nn_{u_2}$. Now since $P$ is not in an OFF position with respect to $B$, neither is $\Nn_{u_2}$. Thus we can choose $t\in (0,\frac{\pi}{n_0})$ such that $\rho_t(\Nn_{u_2})$ becomes in an OFF position with respect to $B$, and at the same time, $P$ also comes in an OFF position with respect to $\rho_t(\Nn_{u_2})$. Let $(\Om_1)_t$ be the domain enclosed by $\rho_t(\Nn_{u_2})$ and $\pa P$, and $(\Om_2)_t=\Om\setminus\overline{(\Om_1)_t}$.

\noi\underline{\bf Case $(i)\; c>0$}: Now using Remark \ref{Final_remark}-$(i)$ for the domain $(\Om_1)_t$ and using Theorem \ref{theorem}-$(i)$ for the domain $(\Om_2)_t$ (as $c>0$), we obtain 
\begin{equation}\label{mono_1}
    \la_1((\Om_1)_t)<\la_1(\Om_1)=\la_2\;\text{and}\; \la_1((\Om_2)_t)<\la_1(\Om_2)=\la_2.
\end{equation}
Let $\psi_1$ and $\psi_2$ be the positive eigenfunction associated to $\la_1((\Om_1)_t)$ and $\la_1((\Om_2)_t)$ respectively. Let $$\A=\{\psi\in W^{1,p}_0(\Om): \psi=a\widehat{\psi}_1+b\widehat{\psi}_2,\; |a|^p||\psi_1||_p^p+|b|^p||\psi_2||_p^p=1\},$$
where $\widehat{\psi}_1$ and $\widehat{\psi}_2$ are the zero extensions of $\psi_1$ and $\psi_2$ respectively to $\Om$. Then obviously $\A\in \E$. Also for any $\psi\in \A$, using \eqref{mono_1}, we get
\begin{align}
    \int\limits_{\Om}|\nabla \psi|^p\dx=|a|^p\int\limits_{(\Om_1)_t}|\nabla \psi_1|^p\dx+|b|^p\int\limits_{(\Om_2)_t}| \nabla\psi_2|^p\dx=\la_1((\Om_1)_t)|a|^p||\psi_1||_p^p+\la_1((\Om_2)_t)|b|^p||\psi_2||_p^p.
\end{align}
Now since $\psi$ was arbitrary, we have $\sup\limits_{\psi\in\A}J(\psi)\leq \max\{\la_1((\Om_1)_t),\la_1((\Om_2)_t)\}<\la_2$, which is a contradiction to the variational characterization \eqref{Variational_second} of $\la_2$. 

\noi\underline{\bf Case $(ii)\; c=0$}: Again by Remark \ref{Final_remark}-$(i)$ for the domain $(\Om_1)_t$, we have 
\begin{equation}\label{eq1}
    \la_1((\Om_1)_t)<\la_1(\Om_1)=\la_2.
\end{equation}
By hypothesis, $\Nn_{u_2}$  and $P$ are concentric. Since $B$ and $P$ are concentric too, so are $B$ and $\Nn_{u_2}$. Thus $\Om_2$ is isometric to $(\Om_2)_t$ and hence
\begin{equation}\label{eq2}
    \la_1((\Om_2)_t)=\la_1(\Om_2)=\la_2.
\end{equation}
Now using \eqref{eq1} and the continuity of $\la_1$ with respect to the perturbations of $(\Om_1)_t$ (cf. \cite[Theorem 1]{Melian2001}), we can choose a domain $A\subsetneq (\Om_1)_t$ such that $\la_1((\Om_1)_t)<\la_1(A)<\la_2.$ Let $B=\Om\setminus \overline{A}$. Then $(\Om_2)_t\subsetneq B$. Therefore \eqref{eq2} yields that $\la_1(B)<\la_1((\Om_2)_t)=\la_2.$ Repeating the arguments of case $(i)$ for domains $A$, $B$ we arrive at a contradiction again. This completes the proof.
\end{proof}
\subsection*{Open Problem}
Note that in Proposition \ref{Monotonicity_lambda}-$(ii)$, we have established the monotonicity of the map $t\longmapsto\la_1(t)$ on $(0,\frac{\pi}{n})$ for $1<p\leq \frac{3}{2}$. However, we anticipate that the strict monotonicity and hence Theorem \ref{Nodal_theorem} hold for this range of $p$ as well.
\section{Acknowledgments} A. M. H. Chorwadwala was supported  by  the  MATRICS:  Science  and  Engineering  Research  Board  Grant MTR/2019/001309. M. Ghosh gratefully acknowledges the financial support provided by IIT Madras. This work was started during the visit of M. Ghosh to IISER Pune. He would like to thank IISER Pune for their invitation and the hospitality provided there.

\bibliographystyle{abbrv} 
\bibliography{reference}

\end{document}